\documentclass[11pt]{article}

\usepackage[utf8]{inputenc}
\usepackage[T1]{fontenc}
\usepackage{mathpazo}
\usepackage{tikz}

\RequirePackage[english]{babel}

\RequirePackage[a4paper,margin=2.5cm]{geometry}
\RequirePackage{parskip}

\newcommand{\affiliation}{\footnote}
\makeatletter
\def\@fnsymbol#1{\ensuremath{\ifcase#1\or *\or \dagger\or \ddagger\or \mathsection\or \|\or **\or \dagger\dagger \or \ddagger\ddagger \else\@ctrerr\fi}}
\makeatother

\usepackage{xcolor}
\definecolor{cblue}{RGB}{0,70,140}
\definecolor{cgreen}{RGB}{100,140,0}
\definecolor{cred}{RGB}{190,10,50}

\usepackage[shortlabels]{enumitem}
\setlist{itemsep=0ex,topsep=0ex,parsep=0.4ex}

\usepackage{amsmath,amsfonts,amssymb,amsthm,mathtools,thmtools,thm-restate,bbm,centernot}

\usepackage[hyphens]{url} 
\usepackage[hidelinks,colorlinks,pagebackref]{hyperref} 
\hypersetup{citecolor=cgreen,linkcolor=cblue,urlcolor=cblue}
\usepackage[capitalise,nameinlink,noabbrev,compress]{cleveref} 

\renewcommand*{\backref}[1]{}
\renewcommand*{\backrefalt}[4]{
	\ifcase #1 Not cited.%
	\or $\uparrow$#2%
	\else $\uparrow$#2%
	\fi%
}

\theoremstyle{plain}
\newtheorem{theorem}{Theorem}[section]
\newtheorem{lemma}[theorem]{Lemma}

\newtheorem{proposition}[theorem]{Proposition}

\theoremstyle{definition}

\makeatletter
\renewenvironment{proof}[1][\proofname]
{\par\pushQED{\qed}
	\normalfont\topsep6\p@\@plus6\p@\relax\trivlist
	\item[\hskip\labelsep\bfseries#1\@addpunct{.}]
	\ignorespaces}
{\popQED\endtrivlist\@endpefalse}
\makeatother

\newcommand{\cH}{\mathcal{H}}
\newcommand{\R}{\mathcal{R}}
\newcommand{\T}{\mathcal{T}}
\newcommand{\cP}{\mathcal{P}}
\newcommand{\G}{\mathcal{G}}
\newcommand{\F}{\mathcal{F}}
\newcommand{\bN}{\mathbb{N}}

\newcommand{\Forb}{{\rm Forb}}

\newcommand{\Rk}{\R_{ind}}

\usepackage[T1]{fontenc}
\usepackage{lmodern}
\usetikzlibrary{decorations.pathreplacing}
\definecolor{pink}{RGB}{185,0,122}
\definecolor{palepink}{RGB}{235,195,231}
\definecolor{cyanframe}{RGB}{132,218,218}
\definecolor{grayvertex}{RGB}{105,105,105}
\definecolor{lavender}{RGB}{226,224,251}
\definecolor{chosen}{RGB}{230,25,35}
\tikzset{vertex/.style={circle,inner sep=1.65pt,fill=grayvertex},
  chosen/.style={vertex,fill=chosen,inner sep=2.1pt},
  attachment/.style={draw=gray!55,line width=.45pt}}
\newcommand{\Fshape}[2]{%
 \fill[#1,rounded corners=12pt] (-.73,-1.48) rectangle (.85,1.50);
 \fill[#2,opacity=.32] (.24,.67) ellipse [x radius=.49,y radius=.72];
 \fill[#2,opacity=.32] (.34,-.08) ellipse [x radius=.45,y radius=.61];
 \fill[#2,opacity=.32] (.23,-1.03) ellipse [x radius=.40,y radius=.32];
 \node[font=\Large] at (-.45,-.18) {$F$};
}
\newcommand{\Fpoints}[2]{%
 \foreach \i/\x/\y in {1/.05/1.05,2/.38/.82,3/.28/.30,4/.48/-.13,5/.27/-.43,6/.12/-.91,7/.35/-1.15}{%
  \coordinate (#1-\i) at (\x,\y);
  \node[vertex,fill=#2] at (#1-\i) {};
 }
}

\title{Vertex-Ramsey equivalence of graphs}
\author{Maria Axenovich and Patrick Schmitteckert \affiliation{Institute of Algebra and Geometry, Karlsruhe Institute of Technology, Germany (\textsf{\href{mailto:maria.aksenovich@kit.edu}{maria.aksenovich@kit.edu}, \textsf{\href{mailto:patrick@schmitteckert.org}{patrick@schmitteckert.org})}}}}
\date{\today}

\begin{document}
    \maketitle

    \begin{abstract}
        For a graph $H$,  its $k$-colour  vertex Ramsey class is the set of all graphs $G$ such that any colouring of the  vertices of $G$ in $k$ colours results in a monochromatic (induced) copy of $H$. 
        We prove that for any $k$,  the $k$-colour vertex Ramsey classes of any  non-isomorphic graphs are distinct. 
       Moreover, we show that this property does not hold when individual graphs are replaced with families of graphs.
            \end{abstract}
    
    \section{Introduction}
  For a graph $H$ and $k\in \bN$, let  the {\it $k$-colour vertex Ramsey class} for $H$, denoted $\R^k(H)$,  be the set of all graphs  $G$, such that for any colouring of the vertices of $G$ in $k$ colours there is a subgraph isomorphic to $H$, that we call a {\it copy of $H$},  with all vertices of the same colour.   Similarly, we define  $\R_{ind}^k(H)$ if the copy of $H$ in the above definition is an induced subgraph of $G$.  
  The fact that  $\Rk^k(H)$ is nonempty for every graph $H$ was first shown by Folkman \cite{Folk}, see also \cite{BR, DHKZ, LRV, NR} for results on vertex Ramsey theory. 
  In this note we prove that any two non-isomorphic graphs have different $k$-colour Ramsey classes, i.e., they are not vertex Ramsey-equivalent. Here, we assume that all graphs considered  have at least one vertex.
 
 \begin{theorem}\label{main}
 For any two   non-isomorphic graphs  $H_1$ and $H_2$ and  any  positive integer $k$,   $\R^k(H_1) \neq \R^k(H_2)$ and $\R_{ind}^k(H_1) \neq \R_{ind}^k(H_2)$.
 \end{theorem}
 
 The study of Ramsey classes is a fundamental area in discrete mathematics addressing quantitative and structural properties of graphs in these classes.
Most existing  research  on Ramsey classes has concerned the setting of  edge colourings. The main motivation of this note is a question by  Fox,  Grinshpun, Liebenau, Person, and Szab\'o \cite{F}, whether any two connected non-isomorphic graphs have distinct edge Ramsey classes. While this question remains open in general,  Theorem \ref{main} answers the  analogous question in the positive  for vertex colourings and all, not only connected, graphs. 

A natural question is whether the main theorem above holds for families of graphs. Let, for a family $\F$ of graphs, $\R^k(\F)$  be the set of all graphs $G$, such that for any colouring of the vertices of $G$ in $k$ colours there is  a monochromatic copy of some  member of $\F$. The induced Ramsey  class $\Rk^k(\F)$ is defined similarly for induced copies. We answer this question in the negative.

\begin{theorem}\label{families}
There are two distinct graph classes $\F$ and $\cH$ such that for any integer $k\geq 2$,  $\R^k(\F)=\Rk^k(\F)=\Rk^k(\cH)=\R^k(\cH).$
\end{theorem}

In Section \ref{sec:notations} we give some notations and needed Ramsey-type properties. Theorem \ref{main} is proved in Section \ref{proof-main}. We prove Theorem \ref{families} in Section \ref{sec:families}. Finally, we make some remarks on extending the main results to hypergraphs in Section \ref{sec:conclusions}.

 \section{Notation and Ramsey properties}\label{sec:notations}
Let $H=(V, E)$ be a graph. We denote its order  $|V(H)|$ by $|H|$, its size $|E(H)|$ by $||H||$, its complement by $\overline{H}$, and 
its independence number by  $\alpha(H)$. If $||H||\geq 1$, let $\beta(H)$ be the largest order of a block of $H$, that is, a maximal 2-connected subgraph or a bridge; if $||H||=0$, let $\beta(H)=1$. For two graphs $G$ and $H$, let $G\vee H$ denote their join, that is a graph obtained as a union of vertex-disjoint copies of $G$ and $H$ together with all possible edges between these two copies.
 We denote the set of all graphs with no induced copy of $H$ by $\Forb(H)$.

 {\bf Construction 1.} Let $H$ be a graph and $q$ be an integer at least $3$.  Consider an  $|H|$-uniform hypergraph $\cH$ of girth at least $q$ and chromatic number at least $k+1$. The existence of such a graph was first shown by Erd\H{o}s and Hajnal \cite{EH}.
 Let $G=G(H, q, k)$ be a graph on vertex set $V(\cH)$ that is a union of $|E(\cH)|$ copies of $H$ induced by respective hyperedges of $\cH$. 
 We have that $G\in \Rk^k(H)$ since any vertex-colouring of $G$ in $k$ colours  and thus of $\cH$ results in a monochromatic hyperedge of $\cH$, i.e., a monochromatic induced copy of $H$. Note that this copy is indeed induced since the girth condition ensures that any two hyperedges of $\cH$ intersect in at most one vertex.

The definition of $\Rk^k$ implies that  $\R_{ind}^k(H_1) \neq \R_{ind}^k(H_2)$ if and only if $\R_{ind}^k(\overline{H}_1) \neq \R_{ind}^k(\overline{H}_2)$.
Using Construction 1 with sufficiently large girth and complementation we see  that for any graph $H$  and any positive integer $k$, there is $G \in \R_{ind}^k(H)$ such that $\alpha(H)=\alpha(G)$.

For a graph $A$, we denote a pairwise vertex-disjoint union of $q$ copies of $A$ by $qA$.  For $k\geq 1$, we say that a graph $H$ has  the  {\it Ramsey repetition property} if for any  integer $q \geq1$ and every graph $A$,   $ qA\in\Rk^k(H)$ if and only if $A\in\Rk^k(H).$
 \begin{lemma}\label{repetition}
For any $k\geq 1$, a  graph has the Ramsey repetition property if and only if it is connected.
\end{lemma}

\begin{proof}
Fix $k\geq 1$.

Assume that $H$ is connected. Let $q\geq1$. If $A\in\Rk^k(H)$, then clearly $q A\in\Rk^k(H)$.
If $A\not\in\Rk^k(H)$, repeat a colouring of $A$ without a monochromatic $H$ 
on its disjoint copies. This implies that $q A\not \in\Rk^k(H)$. Therefore $H$ has a Ramsey repetition property.

Let $H$  be disconnected and have components $H_1,\ldots,H_t$, where $t\geq2$. Let
$L=H_1\vee\cdots\vee H_t$. Choose
$A\in\Rk^k(L)$ with $ \alpha(A)=\alpha(L)$.     Since $\alpha(A)=\alpha(L)=\max_i\alpha(H_i)<\sum_i\alpha(H_i)=\alpha(H)$,
$A$ contains no induced $H$, thus  $A\not\in \Rk^k(H)$. 
Consider  an arbitrary $k$-vertex colouring of $qA$, for $q=k(t-1)+1$.
Each copy of $A$ contains a  monochromatic
$L$, and $t$ of these have the same colour. Taking $H_i$ from the $i$th
of these $t$ distinct copies gives a monochromatic induced $H$.
Therefore $qA \in \Rk^k(H)$.
Thus $H$ has no Ramsey repetition property.
\end{proof}

   \section{Proof of Theorem \ref{main}}\label{proof-main}
   Let $H_1$ and $H_2$ be two non-isomorphic graphs. If $k=1$, we have that $H_i\in \Rk^1(H_i)$, $i\in \{1,2\}$,  thus $\Rk^1(H_1)\neq \Rk^1(H_2)$ and  $\R^1(H_1)\neq \R^1(H_2)$.
   So, we assume that  $k\geq 2$.
   
   We shall first show that $\R^k(H_1) \neq \R^k(H_2)$. If $k=1$, we have that $H_i\in \R^k(H_i)$, thus $\R^k(H_1)\neq \R^k(H_2)$. Let $k\geq 2$.
   Assume first that $|H_1|\neq |H_2|$.  If $n=|H_1|<|H_2|$, let  $G=K_{k(n-1)+1}$.  By pigeonhole principle we see that $G\in \R^k(H_1)$. On the other hand, $G\not\in \R^k(H_2)$ because one can assign $k$ colours to the vertices of $G$ such that each colour class has size at most $n$.
   Now, assume that $|H_1|=|H_2|=n$. Since $H_1$ and $H_2$ are not isomorphic, without loss of generality assume that $H_1$ is not a subgraph of $H_2$.
   Let $G=H_2\vee K_{(k-1)(n-1)}$. To see that  $G\in \R^k(H_2)$, consider an arbitrary vertex-colouring of $G$. At least $n$ vertices are of the same colour, say red. Then the set of $n$ red vertices induces 
   a graph $G''=G'\vee K_{n-\ell}$, for some $\ell\geq 0$, where $G'$ is a subgraph of $H_2$ on $\ell$ vertices. Thus, $G''$ contains $H_2$ as a subgraph.
   On the other hand, by colouring the vertices of $G$ inducing  $H_2$ red and colouring  the remaining vertices in other, $k-1$ colours, such that each of these colours is present on $(n-1)$ vertices,  one has no monochromatic copy of $H_1$.\\
 
 Next, we shall prove that $\R_{ind}^k(H_1) \neq \R_{ind}^k(H_2)$.  Assume otherwise, i.e., that $\R_{ind}^k(H_1) = \R_{ind}^k(H_2)$.
 By Lemma \ref{repetition} we see  that both $H_1$ and $H_2$ are connected or both of them disconnected. By complementation, we can  assume that 
both of them are connected. If one of $H_1$ or $H_2$, say $H_1$, is a single vertex graph, $K_1$,  then $K_1\in \Rk^k(H_1)\setminus \Rk^k(H_2)$, so we are done.  We assume that $H_1$ and $H_2$ are  non-isomorphic connected graphs each containing at least one edge. 

Observe that if $H$ is connected and has an edge, there is an independent set
$I\subseteq V(H)$ such that $\beta(H-I)<\beta(H)$.
Indeed, we can greedily select such an $I$ by starting with some block as a root  in the block-cutvertex tree and processing its blocks away from the root.
Choose one vertex in the root block. For each subsequent block, if its
attachment vertex has already been chosen, choose nothing; otherwise
choose any other vertex in that block. A new vertex has no earlier
neighbours except possibly the attachment vertex, so the chosen set $I$
is independent and meets every block.

 Assume without loss of generality that $\beta(H_1)\geq \beta(H_2)$ and if $\beta(H_1)=\beta(H_2)$, then 
$|H_1|\geq |H_2|$. In particular  $H_2\in \Forb(H_1)$. Choose an independent set $I$ of $H_2$ such that for $F=H_2-I$, 
 $\beta(F)<\beta(H_2)\leq\beta(H_1)$.

 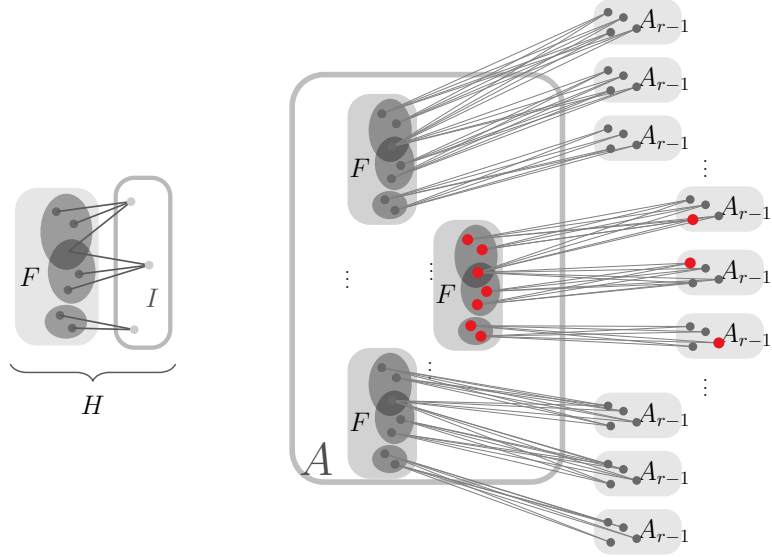
\begin{figure}
\begin{center}
\scalebox{0.7}{

\begin{tikzpicture}[font=\small]
\begin{scope}[shift={(-4.6,0)},scale=1.0]
  \Fshape{gray!20}{black}
  \Fpoints{h}{grayvertex}
  \draw[gray!55,line width=2.3pt,rounded corners=10pt]
    (1.17,-1.52) rectangle (2.20,1.69);
  \node[font=\Large,text=black!70] at (1.86,-.58) {$I$};
  \coordinate (u1) at (1.46,1.24);
  \coordinate (u2) at (1.80,.04);
  \coordinate (u3) at (1.52,-1.18);
  \foreach \i in {1,2,3} {\draw[black!65,line width=.9pt] (h-\i)--(u1);}
  \foreach \i in {3,4,5} {\draw[black!65,line width=.9pt] (h-\i)--(u2);}
  \foreach \i in {6,7} {\draw[black!65,line width=.9pt] (h-\i)--(u3);}
  \foreach \i in {1,2,3} {\node[vertex,fill=black!20] at (u\i) {};}
  \draw[black!65,thick,decorate,decoration={brace,mirror,amplitude=9pt}]
    (-.82,-1.83)--(2.30,-1.83) node[midway,below=13pt,font=\Large,text=black] {$H$};
\end{scope}

\draw[black!25,line width=3pt,rounded corners=17pt]
  (-.1,-4.06) rectangle (5.0,3.64);
\node[font=\Huge,text=black!70] at (.35,-3.63) {$A$};

\begin{scope}[shift={(1.55,2.03)},scale=.83]
  \Fshape{gray!35}{black}
  \Fpoints{a}{grayvertex}
\end{scope}
\begin{scope}[shift={(3.17,-.35)},scale=.83]
  \Fshape{gray!35}{black}
  \Fpoints{b}{chosen}
\end{scope}
\begin{scope}[shift={(1.55,-2.77)},scale=.83]
  \Fshape{gray!35}{black}
  \Fpoints{c}{grayvertex}
\end{scope}

\node at (.94,-.12) {$\vdots$};
\node at (2.53,.03) {$\vdots$};
\node at (2.50,-1.84) {$\vdots$};

\foreach \prefix/\xx/\yone/\ytwo/\ythree in {
  a/6.35/4.64/3.54/2.44,
  b/7.90/1.10/-.10/-1.30,
  c/6.35/-2.80/-3.90/-5.00}{
  \foreach \idx/\yy in {1/\yone,2/\ytwo,3/\ythree}{
    \fill[gray!20,rounded corners=9pt] (\xx-.76,\yy-.43) rectangle (\xx+.90,\yy+.43);
    \node[font=\Large] at (\xx+.56,\yy) {$A_{r-1}$};
    \foreach \j/\dx/\dy in {1/-.49/.18,2/-.20/.08,3/-.44/-.20,4/.05/-.13}{
      \coordinate (\prefix-H\idx-\j) at (\xx+\dx,\yy+\dy);
    }
  }
  \foreach \s in {1,2,3}{\foreach \t in {1,2,3,4}{\draw[black!50] (\prefix-\s)--(\prefix-H1-\t);}}
  \foreach \s in {3,4,5}{\foreach \t in {1,2,3,4}{\draw[black!50] (\prefix-\s)--(\prefix-H2-\t);}}
  \foreach \s in {6,7}{\foreach \t in {1,2,3,4}{\draw[black!50] (\prefix-\s)--(\prefix-H3-\t);}}
  \foreach \idx in {1,2,3}{\foreach \j in {1,2,3,4}{\node[vertex] at (\prefix-H\idx-\j) {};}}
}

\node[chosen] at (b-H1-3) {};
\node[chosen] at (b-H2-1) {};
\node[chosen] at (b-H3-4) {};

\foreach \i in {1,...,7}{
  \node[vertex] at (a-\i) {};
  \node[chosen] at (b-\i) {};
  \node[vertex] at (c-\i) {};
}

\node at (7.70,1.97) {$\vdots$};
\node at (7.70,-2.17) {$\vdots$};
\end{tikzpicture}}
\end{center}
\caption{Construction of the graph $A_r$, $r\geq 2$ and a red copy of $H=H_2$.}
\label{figure}
\end{figure}

 First, we shall construct a graph $A\in \Forb(H_1)\cap \Rk^k(F)$.
 If $F$ has no edges, let $A$ be an edgeless graph on $k(|F|-1)+1$ vertices.
 If $F$ has edges, then $\beta(F)\geq 2$, thus $\beta(H_1)\geq 3$. Let $B$ be a block of $H_1$ on $\beta(H_1)$ vertices.
 Apply Construction 1 to build a graph $A=G(F, ||B||+2, k)$.
Since the girth of the underlying hypergraph is at least $||B||+2$ and $B$ is not contained in copies of $F$ placed in hyperedges of that hypergraph,  $A\in \Forb(B)$, thus $A\in \Forb(H_1)$.   By the property of Construction 1, $A\in \Rk^k(F)$. Thus $A\in \Rk^i(F)$ for all $i\in \{2,\dots,k\}$.

Next,  we recursively define graphs  $A_r$ for $r\in \{1,\dots, k\}$ such that $A_1:=H_2$ and for $r\geq 2$, $A_{r}\in \R_{ind}^{r}(H_2) \setminus \R_{ind}^{r}(H_1)$.
For $r\geq 2$,  we construct the graph $A_r$ based on $A_{r-1}$ as follows. Start with $A$. For every induced copy $\varphi(F)$ of
$F$ in $A$ and every $u\in I$, add a  copy $H_{\varphi,u, r-1}$ of $A_{r-1}$, such that these copies and $A$ are all pairwise vertex-disjoint. 
Make each vertex of $H_{\varphi,u,r-1 }$ adjacent to exactly those vertices of  $\varphi(F)$ that  correspond to the neighbours of $u$ in $H_2$, see Figure \ref{figure}.

We have that $A_r\not\in\Rk^r(H_1)$. Indeed, colour the vertices of $A$ red and the remaining vertices in other $r-1$ colours such that each copy of $A_{r-1}$ has no monochromatic induced copy of $H_1$. This colouring has no induced monochromatic copy of $H_1$.

On the other hand, $A_r\in \Rk^r(H_2)$.  To see this, consider an arbitrary  $r$-colouring of $V(A_r)$.  If some  added copy $H_{\varphi, u, r-1}$ of $A_{r-1}$ uses  at most  $r-1$ colours, it has a monochromatic induced copy of $H_2$. Otherwise, each such copy has all $r$ colours on its vertices.  Choose a monochromatic, say red,  induced  copy $\varphi(F)$ of $F$  in $A$,  and for each $u\in I$ choose a red
vertex of $H_{\varphi,u, r-1}$. Together these vertices and $\varphi(F)$ induce a red copy of $H_2$.

In particular, $A_{k}\in \R_{ind}^{k}(H_2) \setminus \R_{ind}^{k}(H_1)$, implying that  $\Rk^k(H_1)\neq \Rk^k(H_2)$ and  contradicting our assumption.
\qed

\section{Proof of Theorem \ref{families}} \label{sec:families}

%
%
%
%

For a graph $G$, let $t(G)$ be the largest number of pairwise disjoint independent sets of size $3$.
Let  $\F$ and $\cH$ be two families of graphs,  $$\F= \{F: |V(F)|\leq 13, ~|V(F)| -t(F) >8\} ~~\mbox{ and }~~ \cH = \F \cup \{ J\},$$ where 
 $J=K_4\vee K_{3,3}.$ Theorem \ref{families} is an immediate consequence of the following proposition.

\begin{proposition}
The classes $\F$ and $\cH$ are distinct and for any $k\ge2$, 
$
\R^k(\F)=\Rk^k(\F)=\R^k(\cH)=\Rk^k(\cH)=\G$, where $\G=\{G:|V(G)|-t(G)>8k\}.$
\end{proposition}

\begin{proof}
First, we shall show that any vertex colouring of any graph from $\G$ contains a monochromatic induced copy of a graph from $\F$. Let $V_1, \ldots, V_k$ be colour classes of a given vertex colouring of $G\in \G$.  Let $n_i= |V_i|$ and $t_i=t(G[V_i])$, $i\in [k]$.  Since $V_i$'s are pairwise disjoint,  $n_1+ \cdots + n_k = |V(G)|$ and $t_1+\cdots + t_k\leq t(G)$. Thus there is $i\in [k]$ such that $n_i -t_i >8$. If $n_i\leq 13$, $G[V_i]\in \F$. Otherwise, pick any subset $V_i'$ of $V_i$ such that $|V_i'|=13$.
Then $t(G[V_i']) \leq 4$ and $|V_i'| - t(G[V_i']) \geq 13- 4 >8$. Thus $G[V_i']$ is an induced monochromatic subgraph of $G$ that is in $\F$.
Since $\F\subseteq \cH$,   $\G \subseteq \Rk^k(\F) \subseteq  \Rk^k(\F)$. Since an induced subgraph is an ordinary subgraph,  $\G \subseteq \R^k(\cH) \subseteq  \R^k(\cH)$.

Let  $G$ be any graph not from  $\G$.  We shall define its  vertex colouring $c$ with $k$ colours as follows.
Let $t=t(G)$ and $n=|V(G)|$.
We split the vertex set of $G$ into triples and pairs, such that the triples induce independent sets. Formally, let $\T$ be a set of  $t$ pairwise disjoint three-element independent sets and $\cP$ be the set of pairwise disjoint pairs of vertices and maybe one singleton containing all $n-3t$ vertices that are not in sets from $\T$.
We have that $|\T|+|\cP| = t + \lceil (n-3t)/2 \rceil = \lceil (n-t)/2 \rceil \leq 4k$. We shall split $\T\cup \cP$ into $k$ families, making sure that each family gets an appropriate number of triples from $\T$, namely, 4, 3, 1, or 0,  as follows. 
Let $t= 4q + s$, where $q$ and $s$ are integers and $0\leq s\leq 3$.
We have that $q\leq k-s$. First, we shall split $\T$ into  $k$  families, such that the $i$'th family has  $a_i$ triples, $0\leq a_i\leq 4$, $i\in [k]$.\\
\noindent
If $s=0$, let $a_1=\cdots=a_q=4$.\\
If $s\not\in \{0, 2\}$, let $a_1=\cdots=a_q=4$, $a_{q+1}=s$.\\
If $s=2$ and $q\leq  k-2$, let $a_1=\cdots = a_q=4$, $a_{q+1}=a_{q+2}=1$.\\
If $s=2$ and $q= k-1$, let $a_1=\cdots = a_{q-1}=4$, $a_{q}=a_{q+1}=3$.\\
Let the remaining $a_i$'s, if any,  be zero in each of the cases above.
So, we have $a_1+\cdots +a_k = t=|\T|$ and for any $i\in [k]$,  $a_i\neq 2$.
Give all the vertices in the $i$'th family  of triples colour $i$. Furthermore, 
split $\cP$ into $k$ families of sizes at  most $b_i$, where $b_i=4-a_i$, $i\in [k]$.
Since $|\T|+|\cP| \leq 4k$, $b_1+\cdots +b_k =4k - |\T|\geq |\cP|$. Thus this split is possible. 
This gives a $k$-colouring of the vertices of $G$, so that each colour class $V_i$ is a union  of $a_i$ sets from $\T\cup \cP$ and $b_i$ sets from $\cP$.  In addition, no colour class has exactly two sets from $\T$.

Next we shall argue that there is no monochromatic copy of a graph from $\cH$ in $G$ under $c$. 

Assume that there is a monochromatic copy $F$ of a graph from $\F$, say in colour $i$. 
Let $b$ be the number of sets from $\T$ contained in $V(F)$. Then $|V(F)| \leq 3b + 2(4-b)= 8+b$.
Since the members of $\T$ contained in $V(F)$ are independent triples, $t(F)\geq b$.
So, $|V(F)| -t(F) \leq (8+b)-b =8$,  a contradiction to the definition of $\F$.

Assume that there is a  monochromatic copy of $J$, say in color $i$. 
We have that $|V_i|\leq  4a_i + 2(4-a_i)= 8+a_i$ and $\omega(G[V_i]) \leq a_i + 2(4-a_i)=8-a_i$. 
Since $V(J)\subseteq V_i$ and $|V(J)|=10$, we have that $a_i\geq 2$. Since no $a_i$ is equal to $2$, $a_i\geq 3$. Thus  $\omega(G[V_i]) \leq 5$. 
On the other hand $\omega(J) =6$, a contradiction.

Thus, if $G\not\in \G$, then  $G\not \in \R^k(\cH)$.  This implies that $G$ is not in any of the Ramsey classes $ \Rk^k(\cH), \Rk^k(\F),  \R^k(\F)$.
Together with the first statement that if $G\in \G$, then $G$ is in each  of the Ramsey classes $\R^k(\cH),  \Rk^k(\cH), \Rk^k(\F),  \R^k(\F)$, we get
 $\G = \R^k(\F)=\Rk^k(\F)=\R^k(\cH)=\Rk^k(\cH)$.

Finally, we shall show that $\F\neq \cH$ by arguing that $J\not\in \F$. 
Indeed, $t(J)=2$, so $|V(J)|-t(J)=8$.  I.e., $J$ does not satisfy the conditions defining $\F$.
\end{proof}

\section{Concluding remarks}\label{sec:conclusions}
In this note, we proved that any two non-isomorphic graphs have distinct Ramsey classes using in particular such graph parameters as independence number $\alpha$ and the size of a largest block $\beta$, as well as a construction involving high girth high chromatic number hypergraphs.  The proof of this result extends almost verbatim to uniform hypergraphs  by considering respective independence number $\alpha$ and the largest block size of a $2$-shadow graph.

    \textbf{Acknowledgements.}  The proofs were obtained by alternating the authors' and ChatGPT 6.0's arguments.  The paper is written and checked by the authors.
 {\hypersetup{linkcolor=cred}
 \newcommand{\etalchar}[1]{$^{#1}$}


\begin{thebibliography}{99}
 
\bibitem{BR}
J.~I.~Brown and V.~R\"odl,
\newblock A Ramsey type problem concerning vertex colourings,
\newblock {\em J.~Combin. Theory Ser. B} {\bf 52} (1991), 45--52.


\bibitem{DHKZ}
S.~Diskin, I.~Hoshen, M.~Krivelevich, and M.~Zhukovskii,
\newblock On vertex Ramsey graphs with forbidden subgraphs,
\newblock {\em Discrete Math.} {\bf 347} (2024), 113806.


\bibitem{EH}
P.~Erd\H{o}s and A.~Hajnal,
On chromatic number of graphs and set-systems,
Acta Mathematica Academiae Scientiarum Hungaricae
\textbf{17} (1966), 61--99.

\bibitem{Folk}
J.~Folkman,
\newblock Graphs with monochromatic complete subgraphs in every edge coloring,
\newblock {\em SIAM J. Appl. Math.} {\bf 18} (1970), 19--24.

 \bibitem{F} J. Fox, A. Grinshpun, A. Liebenau, Y. Person,  and T. Szab\'o, What is Ramsey-equivalent to a clique?, {\em J.~Combin. Theory Ser. B} 109 (2014), 120--133.

\bibitem{LRV}
T.~\L uczak, A.~Ruci\'nski, and B.~Voigt,
\newblock Ramsey properties of random graphs,
\newblock {\em J.~Combin. Theory Ser. B} {\bf 56} (1992), 55--68.

\bibitem{NR}
J.~Ne\v{s}et\v{r}il and V.~R\"odl,
\newblock Partitions of vertices,
\newblock {\em Comment. Math. Univ. Carolin.} {\bf 17} (1976), 85--95.

\end{thebibliography}
\end{document}